\documentclass[reqno]{amsart}
\pdfoutput=1
\usepackage[T1]{fontenc}
\usepackage{tikz}
\usepackage{microtype}
\usepackage{lmodern}

\usepackage{enumitem}

\newcommand{\be}{\begin{equation}}
\newcommand{\ee}{\end{equation}}
\newcommand{\bea}{\begin{eqnarray}}
\newcommand{\eea}{\end{eqnarray}}
\newcommand{\bean}{\begin{eqnarray*}} 
\newcommand{\eean}{\end{eqnarray*}}

\newcommand{\brray}{\begin{array}}
\newcommand{\erray}{\end{array}}
\newcommand{\ben}{\begin{equation}{nonumber}}
\newcommand{\een}{\end{equation}{nonumber}}

\numberwithin{equation}{section}
\newtheorem{dfn}{Definition}[section]
\newtheorem{thm}[dfn]{Theorem}
\newtheorem{lmma}[dfn]{Lemma}
\newtheorem{ppsn}[dfn]{Proposition}
\newtheorem{crlre}[dfn]{Corollary}
\newtheorem{xmpl}[dfn]{Example}
\newtheorem{rmrk}[dfn]{Remark}
\newtheorem{rmrks}[dfn]{Remarks}
\newcommand{\bdfn}{\begin{dfn}}
\newcommand{\bthm}{\begin{thm}}
\newcommand{\blmma}{\begin{lmma}}
\newcommand{\bppsn}{\begin{ppsn}}
\newcommand{\bcrlre}{\begin{crlre}}
\newcommand{\bxmpl}{\begin{xmpl}}
\newcommand{\brmrk}{\begin{rmrk}}
\newcommand{\brmrks}{\begin{rmrks}}
\newcommand{\edfn}{\end{dfn}}
\newcommand{\ethm}{\end{thm}}
\newcommand{\elmma}{\end{lmma}}
\newcommand{\eppsn}{\end{ppsn}}
\newcommand{\ecrlre}{\end{crlre}}
\newcommand{\exmpl}{\end{xmpl}}
\newcommand{\ermrk}{\end{rmrk}}
\newcommand{\ermrks}{\end{rmrks}}

\newcommand{\IC}{\mathbb{C}}

\newcommand{\IT}{\mathbb{T}}

\foreach \x in {A,...,Z}{\expandafter\xdef\csname \x\endcsname{\noexpand\mathcal{\x}}} 

\foreach \x in {A,...,Z}{\expandafter\xdef\csname frak\x\endcsname{\noexpand\mathfrak{\x}}} 

\newcommand{\cal}{\mathcal}
\newcommand{\cla}{{\cal A}}
\newcommand{\clb}{{\cal B}}
\newcommand{\clc}{{\cal C}}
\newcommand{\cld}{{\cal D}}
\newcommand{\cle}{{\cal E}}
\newcommand{\clf}{{\cal F}}

\newcommand{\clh}{{\cal H}}
\newcommand{\cli}{{\cal I}}
\newcommand{\clj}{{\cal J}}

\newcommand{\cln}{{\cal N}}
\newcommand{\clo}{{\cal O}}

\newcommand{\clq}{{\cal Q}}

\newcommand{\clt}{{\cal T}}

\newcommand{\clv}{{\cal V}}
\newcommand{\clw}{{\cal W}}

\newcommand{\clz}{{\cal Z}}

\def\a*{{\cal A}_{h,*}}
\def\B{{\cal B}(h)}
\def\B1{{\cal B}_1(h)}
\def\b{{\cal B}^{\rm s.a.}(h)}
\def\b1{{\cal B}^{\rm s.a.}_1(h)}

\newcommand{\ot}{\otimes}

\newcommand{\raro}{\rightarrow}

\newcommand{\id}{\mbox{id}}

\newcommand{\flip}{{\rm flip}}
\newcommand{\tensora}{\otimes_{\A}}
\def \qed {$\Box$}
\mathchardef\mhyphen="2D 

\newcommand{\midarrow}{\tikz \draw[-triangle 90] (0,0) -- +(.1,0);}

\def\a*{{\cal A}_{h,*}}
\def\B{{\cal B}(h)}
\def\B1{{\cal B}_1(h)}
\def\b{{\cal B}^{\rm s.a.}(h)}
\def\b1{{\cal B}^{\rm s.a.}_1(h)}

\mathchardef\mhyphen="2D 

\begin{document}

	
	
\title{Levi-Civita connections from toral actions}
\author[Bhattacharjee]{Suvrajit Bhattacharjee}
\address{Mathematical Institute of Charles University, Sokolovsk\'a 83, Prague, Czech Republic}
\email{bhattacharjee@karlin.mff.cuni.cz}
\author[Joardar]{Soumalya Joardar}\address{Department of Mathematics and Statistics\\ IISER Kolkata, Mohanpur-741246, West Bengal\\
	India}
\email{soumalya@iiserkol.ac.in}
\author[Mukhopadhyay]{Sugato Mukhopadhyay}\address{Institute of Mathematics of the Polish Academy of Sciences\\ ul. \'Sniadeckich 8\\ 00-656 Warszawa\\ Poland}
\email{m.xugato@gmail.com\\ smukhopadhyay@impan.pl}



\begin{abstract}
We construct tame differential calculi coming from toral actions on a class of $\textup{C}^*$\nobreakdash-algebras. Relying on the existence of a unique Levi-Civita connection on such a calculus, we prove a version of the Bianchi identity. A Gauss-Bonnet theorem for the canonical tame calculus of rank two is studied.
\end{abstract}



\maketitle

\section{Introduction}
The notions of connection and curvature are important in any form of geometry. In the realm of noncommutative geometry, quite a few notions of connections and curvatures have been introduced. In the setup of noncommutative geometry due to A. Connes, curvature and metric are obtained directly by looking at the asymptotic expansion of a suitable Laplacian. In more algebraic setups (see \cite{Arn, Majid, moscovici, Breggs, peterka, khalkhali1, koszul}) connections are defined on the bimodules of one-forms (also see \cite{rosenberg} where connections are defined on the space of vector fields). Metric is also defined suitably. Then the main challenge is to prove the existence and uniqueness of a Levi-Civita connection compatible with a suitable class of metric.

In recent years, a particular class of differential calculus over complex algebras has been introduced and studied (see \cite{article1, Bhowmick_gos_joardar, conformal_bhowmick}). Also a notion of pseudo-Riemannian metric has been given over such differential calculus called the tame calculus. The main virtue of such a class is that they admit a unique Levi-Civita connection corresponding to each of a large class of pseudo-Riemannian metrics called strongly $\sigma$\nobreakdash-compatible metrics (for definition of such a metric and the Levi-Civita connection see \cite{Bhowmick_gos_joardar}). Once the existence and uniqueness of such a connection is established, quite naturally various geometric quantities like scalar curvature, Ricci tensor, etc. have been subsequently defined and studied for large class of examples. In this paper we prove the existence of such a tame calculus over a class of $\textup{C}^*$\nobreakdash-algebras where the tame calculus is built quite naturally from toral actions over such $\textup{C}^*$\nobreakdash-algebras. Examples cover noncommutative $n$\nobreakdash-torus, group $\textup{C}^*$\nobreakdash-algebras over finitely generated groups and Cuntz algebra. The tame calculus over group $\textup{C}^*$\nobreakdash-algebra is shown to be bicovariant with respect to its natural coproduct. In all the examples the bimodule of one\nobreakdash-forms has been a free module with the exterior derivative killing the basis elements. These examples also admit a canonical bilinear metric denoted by $g_{0}$. Such a tame calculus automatically admits a unique Levi-Civita connection corresponding to any strongly $\sigma$\nobreakdash-compatible metric. We prove a version of the Bianchi identity which holds for all the examples covered in this paper.

In the penultimate section we formulate a Gauss-Bonnet theorem for the canonical calculus of rank two obtained in this paper on $\textup{C}^*$\nobreakdash-algebras admitting a $\mathbb{T}^{2}$\nobreakdash-action in a suitable sense. The formulation is in the spirit of A. Connes and P. Tretkoff (see \cite{tretkof}). There the version of Gauss-Bonnet theorem for noncommutative torus is in the set up of spectral triples, where the integrated scalar curvature of the noncommutative torus is obtained directly from the spectral data. As mentioned earlier the metric information is also encoded in the spectral data. The proposed version is the following:
\begin{align*}
&\textit{The integrated scalar curvature of the noncommutative two-torus is independent}\\
&\textit{of a deformation parameter.}
\end{align*}

A number of works followed then (see \cite{Ponge, khalkhali2, Sitarz, Khalkhali}). The analogous statement of A. Connes' proposed Gauss-Bonnet theorem in the tame calculus set up would be that the integrated scalar curvature of the canonical calculus of rank two is independent of the positive deformation parameter $k$ of conformally deformed metric $kg_{0}$, where $g_{0}$ is the canonical bilinear metric on the tame calculus. It is shown that if one assumes the Gauss-Bonnet theorem in the presence of a tracial state the integrated scalar curvature is forced to be zero for all smooth deformation parameters (see point (3) of Remark \ref{wedge_extension} for the definition of a smooth deformation parameter in our set up). Then the Gauss-Bonnet theorem is proved for the calculus of rank two of noncommutative $2$\nobreakdash-torus and the group $\textup{C}^*$\nobreakdash-algebra for free group on two generators. The proof is essentially due to Rosenberg (\cite{rosenberg}). Then we go beyond the conformally deformed metric and consider a new one parameter class of strongly $\sigma$\nobreakdash-compatible metrics and show that the analogous Gauss-Bonnet theorem fails for noncommutative $2$\nobreakdash-torus rendering the conformally deformed metrics a special place from the point of view of the Gauss-Bonnet theorem. We end the paper by discussing some future directions. In summary let us mention the new contributions of this paper in this widely studied area:
  \begin{itemize}
  	\item In Theorem \ref{main}, the construction of the tame calculus on the noncommutative $2$\nobreakdash-torus is generalized to produce tame calculi on a reasonable class of $\textup{C}^*$\nobreakdash-algebras. The class in particular includes the Cuntz algebras. In this paper, the main tool to construct such a tame calculus of rank $n$ on the Cuntz algebra $\mathcal{O}_n$ is the action of the torus $\mathbb{T}^n$. In \cite{Cuntz}, for $\mathcal{O}_3$, instead of the action of $\mathbb{T}^3$, a canonical action of $\textup{SO}(3)$ was considered. Since a centered bimodule, generated freely by $n$\nobreakdash-elements, over a fixed algebra is unique up to bimodule isomorphisms, for $n = 3$ the bimodule of one-forms of rank three constructed in the present paper is isomorphic to the one in [16]. However, the respective exterior derivatives are different, which lead to non-equivalent differential calculi. Also, in the present article the Christoffel symbols are all zero for the Levi-Civita connection corresponding to the canonical bilinear metric and consequently, the scalar curvature is zero too for $\clo_{n}$ for all $n$, whereas the scalar curvature obtained for $\clo_{3}$ in \cite{Cuntz} is $-\frac{3}{4}$ (see Theorem 4.7 of \cite{Cuntz}). 
  	\item A version of the Bianchi identity (Equation \eqref{Bianchiidentity}) has been deduced.
  	\item Assuming that the Gauss-Bonnet theorem holds for a tame calculus of rank two on a $\textup{C}^*$\nobreakdash-algebra with a {\bf tracial} state, in Lemma \ref{13jul21sm1}, the integrated scalar curvature is shown to be necessarily zero for all smooth deformation parameters.
  	\item In Corollary \ref{13jul21sm2}, the Gauss-Bonnet theorem is shown to hold for the group $\textup{C}^*$\nobreakdash-algebra of the free group with two generators (extending the noncommutative $2$\nobreakdash-torus example).
  	\item In Subsection \ref{13jul21sm3}, a new one parameter class of strongly $\sigma$\nobreakdash-compatible metrics on noncommutative $2$\nobreakdash-torus is presented for which the Gauss-Bonnet theorem fails.
  	\end{itemize}     

	  \subsection*{Acknowledgments} The first author is supported by the Charles University PRIMUS grant \textit{Spectral Noncommutative Geometry of Quantum Flag Manifolds} PRIMUS/21/SCI/026. This work started while the first author was visiting the second author in February, 2020. He wishes to thank the second author for the invitation and kind hospitality. He is also grateful to Indian Statistical Institute, Kolkata and Prof. Debashish Goswami for offering him a visiting scientist position, where this work was written up. The second author thanks the Department of Science and Technology, India (DST/INSPIRE/04/2016/002469). The third author was partially supported by the National Science Center of Poland (NCN) grant no. 2020/39/I/ST1/01566. All the authors would like to thank Jyotishman Bhowmick for many useful discussions. We are also grateful to the referee for her/his careful reading and suggestions for improvements.

\section{Preliminaries}

\subsection{Tame calculus} In this subsection, we recall the definition of a tame differential calculus on an algebra $\cla$ and state the theorem guaranteeing existence of a unique Levi-Civita connection. We remark that although we shall consider $\textup{C}^*$\nobreakdash-algebras ($\ast$\nobreakdash-algebra in general), the $\ast$\nobreakdash-structure does not have any particular role. In particular, we shall not consider $\ast$\nobreakdash-compatibility of metric and connections. Moreover, in case of $\textup{C}^*$\nobreakdash-algebras, the tame calculus will be constructed on a canonical dense $\ast$\nobreakdash-subalgebra which we shall not mention explicitly.

\bdfn \label{diffcal}
A differential calculus on a $\mathbb{C}$\nobreakdash-algebra $\A$ is a differential graded algebra $ (  \Omega ( \A ), \wedge, d  ) $ generated \footnote{as a differential graded algebra} by elements of degree zero. 
\edfn

In this paper, the bimodule of one-forms of a generic differential calculus will be denoted by $\Omega^{1}(\cla).$ Now we recall the notion of a quasi-tame calculus as introduced in \cite{Bhowmick_gos_joardar}. 

\bdfn \cite[Definition 2.11]{article1}\label{quasitame}
A differential calculus $ ( \Omega ( \cla ), \wedge, d )$ is said to be quasi-tame if the following conditions hold:
\begin{enumerate}
	\item The bimodule $\Omega^{1}(\cla)$ is finitely generated and projective as a right $\cla$ module.
	\item The following short exact sequence of right $\cla$\nobreakdash-modules splits:
	$$ 0 \rightarrow {\rm Ker} ( \wedge ) \rightarrow \Omega^{1}(\cla) \ot_{\cla} \Omega^{1}(\cla) \rightarrow \Omega^2 ( \cla ) \rightarrow 0. $$
	In particular, there exists a right $\cla$\nobreakdash-module $\clf$ isomorphic to the bimodule of two\nobreakdash-forms $\Omega^{2}(\cla)$ such that:
	\begin{equation} \label{splitting25thmay2018}
	\Omega^{1}(\cla) \ot_{\cla} \Omega^{1}(\cla) = {\rm Ker}(\wedge) \oplus \clf 
	\end{equation}
	\item The idempotent $P_{\rm sym} \in {\rm Hom}_{\cla}(\Omega^{1}(\cla) \ot_{\cla} \Omega^{1}(\cla), \Omega^{1}(\cla) \ot_{\cla} \Omega^{1}(\cla))$ with range  ${\rm Ker}(\wedge)$ and  kernel $\clf$ is an $\cla$\nobreakdash-bimodule map. 
\end{enumerate}
\edfn

We denote the map $2P_{\rm sym}-1$ by $\sigma$. We recall the definition of a pseudo-Riemannian metric on a quasi-tame differential calculus.

\bdfn [\cite{article1,Bhowmick_gos_joardar}]\label{metricdefn}
Let $ ( \Omega ( \cla ), \wedge, d )$ be a quasi-tame differential calculus. A pseudo-Riemannian metric $ g $ on $ \Omega^{1}(\cla) $ is
an element of $ {\rm Hom}_{\cla} ( \Omega^{1}(\cla) \ot_{\cla} \Omega^{1}(\cla), \cla ) $ such that
\begin{enumerate}
	\item $g$ is symmetric, i.e., $ g \circ \sigma = g $;
	\item $g$  is non-degenerate, i.e., the right $ \cla$\nobreakdash-linear map $ V_g: \Omega^{1}(\cla) \rightarrow {\Omega^{1}(\cla)}^* $ defined by $ V_g ( \omega ) ( \eta ) = g ( \omega \otimes_{\cla} \eta ) $ is
	an isomorphism of right $ \cla$\nobreakdash-modules, where $\Omega^{1}(\cla)^{\ast}$ stands for the right $\cla$\nobreakdash-module ${\rm Hom}_{\cla}(\Omega^{1}(\cla),\cla)$.
\end{enumerate}
\edfn

We write $\clz(\Omega^{1}(\cla))$ and $\clz(\cla)$ for the center of the bimodule $\Omega^{1}(\cla)$ and the algebra $\cla$, respectively.

\bdfn [\cite{Bhowmick_gos_joardar}]\label{tame}
A quasi-tame differential calculus (Definition \ref{quasitame}) $ (\Omega^{1}(\cla), d ) $ is said to be tame if
\begin{enumerate}
	\item The map $u^{\Omega^{1}(\cla)}:\clz(\Omega^{1}(\cla)) \otimes_{\clz(\cla)} \cla \rightarrow \Omega^{1}(\cla)$ defined by
	$$u^{\Omega^{1}(\cla)}(\sum_i e_i \otimes_{\clz(\cla)} a_i)=\sum_i e_i a_i$$
	is an isomorphism of vector spaces,
	\item $\sigma := 2 P_{\rm sym} - 1$ satisfies the following equation for all $\omega, \eta \in \clz(\Omega^{1}(\cla)):$
	\begin{equation} \label{17thdec20191} \sigma ( \omega \tensora \eta ) = \eta \tensora \omega. \end{equation}
\end{enumerate}
\edfn

Now we recall the definition of a strongly $\sigma$\nobreakdash-compatible pseudo-Riemannian metric on a tame calculus which is equivalent to but slightly different from the original definition. We take the equivalent criterion of Proposition 4.2 of \cite{Bhowmick_gos_joardar} as the definition of a strongly $\sigma$\nobreakdash-compatible metric.

\bdfn\label{sigmacompatible}
A pseudo-Riemannian metric $g$ on a tame differential calculus is said to be strongly $\sigma$\nobreakdash-compatible if any two elements of the set $\{g(\omega\ot\eta):\omega,\eta\in\clz(\Omega^{1}(\cla))\}$ commute. 
\edfn  

\brmrk \label{metricalgbera}
Let $g$ be a strongly $\sigma$\nobreakdash-compatible pseudo-Riemannian metric on a tame differential calculus $(\Omega^{1}(\cla), d)$ where $\Omega^{1}(\cla)$ is a free bimodule over $\A$. Let us denote the matrix $((g(e_i \otimes_{\A} e_j)))_{ij}$ by $G$, where $(e_i)_i$ is a fixed ordered central basis of $\Omega^{1}(\cla)$. Let $\A_G$ denote the algebra generated by the entries of $G$. Since $g$ is strongly $\sigma$\nobreakdash-compatible, $\A_G$ is a commutative subalgebra of $\A$. Hence it makes sense to talk about $\mathrm{det}(G)$ as an element of $\A_G$ and $\mathrm{adj}(G)$ (the adjugate of $G$) as a matrix of the same order as $G$ and with entries from $\A_G$. Moreover, $G \mathrm{adj}(G) = \mathrm{adj}(G) G = \mathrm{det}(G)$. Note that if $G$ is an invertible element of $M_n(\A)$ then $\det(G)$ is an invertible element in $\A$. Conversely, if $\mathrm{det}(G)$ is an invertible element of $\A_G$, then $G$ in an invertible element of $M_n(\A_G)$ and hence of $M_n(\A)$. Moreover, the entries of $G^{-1}$ being from $\A_G$, they commute with each other and with the entries of $G$. Let us denote $G^{-1}_{ij}$ by $g^{ij}$. Metrics such that $G$ is invertible have been considered in \cite{Ponge}. Note that the definition of strongly $\sigma$\nobreakdash-compatible metrics does not include invertibility of $G$. But for the examples of strongly $\sigma$\nobreakdash-compatible metrics considered in this paper, the matrix $G$ is always invertible and hence the following general results will be applicable for the metrics considered in this paper. We would like to mention that the self-compatible metric in the sense of \cite{Ponge} is strongly $\sigma$\nobreakdash-compatible (a proof can be found in \cite{conformal_bhowmick}). 
\ermrk
In this paper, we shall be concerned with tame differential calculi $(\Omega(\cla),\wedge,d)$ arising from a spectral data (see Subsection \ref{Connes_forms}) such that the module $\Omega^{1}(\cla)$ is free as a bimodule with a central basis $e_{1},\dots,e_{n}$. Note that in such a situation, adapting the proof of Theorem 3.4 of \cite{Cuntz}, it can be shown that the conditions of a tame calculus are satisfied. On such a tame calculus, the following are easily verifiable examples of strongly $\sigma$\nobreakdash-compatible metrics.

\bxmpl
The canonical bilinear metric $g_{0}$: $g_{0}(\sum_{i,j} e_{i}\ot e_{j}a_{ij})=\sum_{i}a_{ii}.$
\exmpl 

\bxmpl
The conformally deformed metric: for an invertible element $k\in\cla$, the conformally deformed metric $kg_{0}(\sum_{i,j} e_{i}\ot e_{j}a_{ij})=\sum_{i}ka_{ii}.$
\exmpl 

\bxmpl\label{new} For an invertible element $k\in\cla$, the metric $g$ given by \[g(\sum_{i,j} e_{i}\ot e_{j}a_{ij})=ka_{11}+\sum_{i=2}^{n}a_{ii}.\]
\exmpl
\bdfn
\label{connexion}
Suppose $(\Omega(\cla),\wedge,d)$ is a differential calculus on $\cla$. A (right) connection on an $\cla$-bimodule $\cle$ is a $\mathbb{C}$-linear map $\nabla:\cle\raro \cle\ot_{\cla}\Omega^{1}(\cla)$ such that\begin{displaymath}
\nabla(ea)=\nabla(e)a+e\ot_{\cla} da.
\end{displaymath} 
The torsion of a connection $\nabla$ on the bimodule $\cle:=\Omega^{1}(\cla)$ is the right $\cla$-linear map $T_{\nabla}:=\wedge \circ \nabla+d:\Omega^{1}(\cla)\raro\Omega^{2}(\cla)$. The connection $\nabla$ is said to be torsionless if $T_{\nabla}=0$.
\edfn 
It is a fact (see Theorem 3.3 of \cite{koszul}) that given a tame calculus there is always a torsionless connection on $\Omega^{1}(\cla)$.
\brmrk
In this paper, we shall only consider right connections on the bimodule $\Omega^{1}(\cla)$ and will do so without mentioning it explicitly from now on.
\ermrk 
\begin{ppsn} \cite[Subsection 4.1]{article1} \label{10jun22sm2}
	For a connection $\nabla$ and a pseudo-Riemannian metric $g$ on $\Omega^1(\A)$, let us define
	\[ \Pi^0_g(\nabla) : \Z(\Omega^1(\A)) \otimes_{\IC} \Z(\Omega^1(\A)) \to \Omega^1(\A) \]
	by the map given by
	\[ \Pi^0_g(\nabla) (\omega \otimes_{\IC} \eta) = (g \tensora \id) \sigma_{23}(\nabla(\omega) \tensora \eta + \nabla(\eta) \tensora \omega). \]
	The map $\Pi^0_g(\nabla)$ is extended to a well-defined $\IC$-linear map
	\[  \Pi_g(\nabla) : \Omega^1(\A) \tensora \Omega^1(\A) \to \Omega^1(\A) \]
	where, for $\omega$ and $\eta$ in $\Z(\Omega^1(\A))$ and $a$ in $\A$, the extension is given by
	\begin{equation} \label{10jun22sm1}
		\Pi_g(\nabla) (\omega \tensora \eta a) = \Pi^0_g (\nabla) (\omega \otimes_{\IC} \eta) a + g(\omega \tensora \eta) d(a).
	\end{equation}
\end{ppsn}

\begin{dfn}
	We say that a torsionless connection $\nabla$ on $\Omega^1(\A)$ is a Levi-Civita connection compatible with a pseudo-Riemannian metric $g$ if
	\[ \Pi_g(\nabla) = d \circ g . \]
\end{dfn}

The following theorem is the most important aspect of a tame calculus. 
\bthm {\cite[Theorem 4.4]{Bhowmick_gos_joardar}} \label{13jul21sm5}
Let $(\Omega^{1}(\cla),d,\wedge)$ be a tame differential calculus and $g$ be a strongly $\sigma$\nobreakdash-compatible metric. Then in the presence of a pseudo-Riemannian bilinear metric $g_{0}$, there exists a unique Levi-Civita connection for the triple $(\Omega^{1}(\cla),d,g)$.
\ethm
\bdfn
Let $\Omega^{1}(\cla)$ be a free bimodule with a central basis $\{e_{1},\ldots,e_{n}\}$. Then one can define the Christoffel symbols $\Gamma^{i}_{jk}$ of a connection $\nabla$ as follows:
\begin{displaymath}
\nabla(e_{i})=\sum_{j,k}e_{j}\ot e_{k}\Gamma^{i}_{jk}.
\end{displaymath}
\edfn 
We now state and prove a theorem which explicitly determines the Christoffel symbols of the unique Levi-Civita connection obtained in the Theorem \ref{13jul21sm5}.

\begin{thm}
	\label{christoffelgeneral}
	Let $\nabla_g$ be the Levi-Civita connection on the data $(\Omega^{1}(\cla), d, g)$, where $(\Omega^{1}(\cla), d)$ is a tame differential calculus, $\Omega^{1}(\cla)$ is a free bimodule over $\A$ with central basis $\{e_{1},\ldots, e_{n}\}$ and $g$ is a strongly $\sigma$\nobreakdash-compatible metric on $\Omega^{1}(\cla)$ with $G$ invertible. Moreover, let $\nabla_0$ be a torsionless connection, and $\partial_i$ be the derivations such that $d(a)=\sum_{i=1}^{n}\partial_{i}(a)e_{i}$ for all $a$ in $\cla$. Then the Christoffel symbols $\Gamma^i_{jk}$ of $\nabla_g$ are given by
	\begin{equation} \label{10feb21sm15}
	\begin{aligned}
	\Gamma^p_{ml} = & \frac{1}{2}\Big( \sum_j g^{lj} \partial_m (g_{pj}) + \sum_i g^{mi} \partial_l (g_{ip}) -\sum_{ijn} g_{pn} g^{li} g^{mj} \partial_n (g_{ij}) \Big) \\
	& + \frac{1}{2} \Big( (\Gamma_0)^p_{ml} - (\Gamma_0)^p_{lm} \Big) 
	+ \frac{1}{2} \sum_{in} g_{pn}g^{mi} \Big( (\Gamma_0)^i_{ln} - (\Gamma_0)^i_{nl} \Big) \\
	& + \frac{1}{2} \sum_{in} g_{pn} g^{li} \Big( (\Gamma_0)^i_{mn} - (\Gamma_0)^i_{nm} \Big),
	\end{aligned}
	\end{equation}
	where $(\Gamma_{0})^{i}_{jk}$ are the Christoffel symbols for the torsionless connection $\nabla_{0}$.
\end{thm}
\begin{proof}
	Let $(e_i)_i$ be our choice of ordered central basis of $\Omega^{1}(\cla)$. Since the Christoffel symbols of $\nabla_g$ are given by $\Gamma^i_{jk}$, we have that $\nabla_g(e_i) = \sum_{jk} e_j \tensora e_k \Gamma^i_{jk}.$ Similarly, we have $\nabla_{0}(e_{i})=\sum_{jk}e_{j}\ot e_{k}(\Gamma_0)^i_{jk}$. Moreover, let us define a map $L: \Omega^{1}(\cla) \to \Omega^{1}(\cla) \tensora \Omega^{1}(\cla)$ by $L := \nabla_g - \nabla_0$. Since the space of (right) connections on $\Omega^{1}(\cla)$ is an affine space, $L$ is a right $\A$\nobreakdash-linear map. We denote its coefficients by $L^i_{jk}$. Then we have that
	\begin{equation} \label{10feb21sm4}
	L^i_{jk} = \Gamma^i_{jk} - (\Gamma_0)^i_{jk}.
	\end{equation}
	Since $\nabla_g$ and $\nabla_0$ are both torsionless, we have that
	\[ \wedge \circ L = \wedge \circ (\nabla_g - \nabla_0) = -d + d = 0. \]
	Since $\wedge (L(e_i)) = \sum_{j \le k} e_j \wedge e_k (L^i_{jk} - L^i_{kj})$, we have that for all $i, j, k$,
	\begin{equation} \label{10feb21sm5}
	L^i_{jk} = L^i_{kj}.
	\end{equation}
	\indent Recall from \eqref{10jun22sm1}, the definition of $\Pi_g(\nabla)$. Moreoever, define $\Phi_g(L) = \Pi_g(\nabla_g) - \Pi_g(\nabla_0)$. Since $\nabla_g$ is compatible with $g$ we have that
	\begin{equation} \label{10feb21sm6}
	\Phi_g(L) = \Pi_g(\nabla_g) - \Pi_g(\nabla_0) = d \circ g - \Pi_g(\nabla_0).
	\end{equation}
	Applying \eqref{10feb21sm6} on $e_i \tensora e_j$, we get
	\begin{equation} \label{10feb21sm7}
	\begin{aligned}
	&\sum_{kl} (g \tensora \id) \sigma_{23}(e_k \tensora e_l \tensora e_j L^i_{kl} + e_k \tensora e_l \tensora e_i L^j_{kl})\\
	= \ &d(g_{ij}) - (g \tensora \id)\sigma_{23}(e_k \tensora e_l \tensora e_j (\Gamma_0)^i_{kl} + e_k \tensora e_l \tensora e_i (\Gamma_0)^j_{kl}).
	\end{aligned}
	\end{equation} \label{10feb21sm8}
	Simplifying \eqref{10feb21sm7}, we get
	\begin{equation} \label{10feb21sm9}
	\sum_l e_l \sum_k \big( g_{kj} L^i_{kl} + g_{ki} L^j_{kl} \big)
	= \sum_l e_l \big(\partial_l(g_{ij}) - \sum_k (g_{kj} (\Gamma_0)^i_{kl} + g_{ki} (\Gamma_0)^j_{kl}) \big).
	\end{equation}
	Next we collect the coefficients of $e_l$ from \eqref{10feb21sm9} to get the equation
	\begin{equation} \label{10feb21sm10}
	\sum_k g_{kj} L^i_{kl} = - \sum_k g_{ki} L^j_{kl} + \partial_l(g_{ij}) - \sum_k (g_{kj} (\Gamma_0)^i_{kl} + g_{ki} (\Gamma_0)^j_{kl}).
	\end{equation}
	Applying $g^{mj}$ to both sides of \eqref{10feb21sm10} and summing over the index $j$, we get the following identity for all $i, l, m$:
	\begin{equation} \label{10feb21sm11}
	\begin{aligned}
	L^i_{ml} \big( = \sum_{jk} g^{mj} g_{kj} L^i_{kl} \big) = & - \sum_{jk} g^{mj}g_{ki} L^j_{kl} + \sum_j g^{mj} \partial_l(g_{ij})\\
	& - (\Gamma_0)^i_{ml} - \sum_{jk} g^{mj} g_{ki} (\Gamma_0)^j_{kl}.
	\end{aligned}
	\end{equation}
	We have used the facts that $g_{ij} = g_{ji}$ for all $i, j$ and $\sum_j g^{ij} g_{jk} = \delta_{ik}$ for all $i, k$ while obtaining the above identity. We will be using these and the fact $g^{ij} = g^{ji}$ often in the proof of the current theorem.\\
	\indent Recall from \eqref{10feb21sm5} that	$L^i_{ml} = L^i_{lm}$ for all $i,l,m$. Then, swapping the indices $l$ and $m$ in the right hand side of \eqref{10feb21sm11}, we get
	\begin{equation} \label{10feb21sm12}
	\begin{aligned}
	L^i_{ml} \big( = L^i_{lm} \big) = & - \sum_{jk} g^{lj}g_{ki} L^j_{km} + \sum_j g^{lj} \partial_m(g_{ij})\\
	& - (\Gamma_0)^i_{lm} - \sum_{jk} g^{lj} g_{ki} (\Gamma_0)^j_{km}.
	\end{aligned}
	\end{equation}
	Next we apply $g^{ni}$ on both sides of \eqref{10feb21sm12} and sum over the index $i$ to arrive at
	\begin{equation}
	\begin{aligned}
	\sum_i g^{ni} L^i_{ml} = & - \sum_{ijk} g^{ni} g^{lj} g_{ki} L^j_{km} + \sum_{ij} g^{ni} g^{lj} \partial_m (g_{ij}) \\
	& - \sum_i g^{ni} (\Gamma_0)^i_{lm} - \sum_{ijk} g^{ni} g^{lj} g_{ki} (\Gamma_0)^j_{km}.
	\end{aligned}
	\end{equation}
	Using the fact that $g^{ij}$ and $g_{kl}$ commute for all $i,j,k,l$ as well as the previously discussed properties of $g$, we simplify the above equation to get
	\begin{equation} \label{10feb21sm13}
	\begin{aligned}
	& \sum_i \big( g^{ni} L^{i}_{ml} + g^{li} L^i_{nm} \big) \\
	= & \sum_{ij} g^{ni} g^{lj} \partial_m (g_{ij}) - \sum_i \big(g^{ni} (\Gamma_0)^i_{lm} - g^{li} (\Gamma_0)^i_{nm} \big).
	\end{aligned}
	\end{equation}
	Let us now observe the left hand side of \eqref{10feb21sm13}. The expressions $\sum_i g^{ni} L^i_{ml}$ and $\sum_i g^{li}L^i_{nm}$ are related by a cyclic permutation of $(l,m,n)$. Hence, permuting $l$, $m$ and $n$ in the expression of \eqref{10feb21sm13}, we obtain three distinct equations in three symbolic unknowns $\sum_i g^{(n)i} L^i_{(m)(l)}$, where $((l),(m),(n))$ are cyclic permutations of $(l,m,n)$. This is explicitly solvable for $\sum_i g^{ni} L^i_{ml}$ and we get
	\begin{equation} \label{10feb21sm14}
	\begin{aligned}
	\sum_i g^{ni} L^i_{ml} = & \frac{1}{2} \sum_{ij} \Big( g^{ni} g^{lj} \partial_m(g_{ij}) + g^{mi} g^{nj} \partial_l(g_{ij}) - g^{li} g^{mj} \partial_n (g_{ij}) \Big) \\
	& - \frac{1}{2} \sum_i\Big( g^{ni} (\Gamma_0)^i_{lm} + g^{mi} (\Gamma_0)^i_{nl} - g^{li} (\Gamma_0)^i_{mn} \Big) \\
	& - \frac{1}{2} \sum_i \Big( g^{li} (\Gamma_0)^i_{nm} + g^{ni} (\Gamma_0)^i_{ml} - g^{mi} (\Gamma_0)^i_{ln} .\Big)
	\end{aligned}
	\end{equation}
	We further simplify \eqref{10feb21sm14} by applying $g_{pn}$ to both sides and summing over the index $n$ to get
	\begin{equation}
	\begin{aligned}
	L^p_{ml} = & \frac{1}{2} \Big( \sum_j g^{lj} \partial_m(g_{pj}) + \sum_i g^{mi} \partial_l(g_{ip}) - \sum_{ijn} g_{pn} g^{li}g^{mj} \partial_n(g_{ij}) \Big)\\
	& - \frac{1}{2} \Big( (\Gamma_0)^p_{lm} + \sum_{in} g_{pn} g^{mi} (\Gamma_0)^i_{nl} - \sum_{in} g_{pn} g^{li} (\Gamma_0)^i_{mn} \Big) \\
	& - \frac{1}{2} \Big( \sum_{in} g_{pn} g^{li} (\Gamma_0)^i_{nm} + (\Gamma_0)^p_{ml} - \sum_{in} g_{pn} g^{mi} (\Gamma_0)^i_{ln} \Big).
	\end{aligned}
	\end{equation}
	Recalling, then, from (\ref{10feb21sm4}) that $\Gamma^p_{ml} = (\Gamma_0)^p_{ml} + L^p_{ml}$ and simplifying further, we finally get
	\begin{equation}
	\begin{aligned}
	\Gamma^p_{ml} = & \frac{1}{2}\Big( \sum_j g^{lj} \partial_m (g_{pj}) + \sum_i g^{mi} \partial_l (g_{ip}) -\sum_{ijn} g_{pn} g^{li} g^{mj} \partial_n (g_{ij}) \Big) \\
	& + \frac{1}{2} \Big( (\Gamma_0)^p_{ml} - (\Gamma_0)^p_{lm} \Big) 
	+ \frac{1}{2} \sum_{in} g_{pn}g^{mi} \Big( (\Gamma_0)^i_{ln} - (\Gamma_0)^i_{nl} \Big) \\
	& + \frac{1}{2} \sum_{in} g_{pn} g^{li} \Big( (\Gamma_0)^i_{mn} - (\Gamma_0)^i_{nm} \Big),
	\end{aligned}
	\end{equation} which was to be obtained.
\end{proof}

\subsection{Tame calculus on a Hopf algebra} Let $(\cla,\Delta)$ be a Hopf algebra such that there is a tame calculus $(\Omega(\cla),\wedge,d)$ on it. Then it is natural to expect bicovariance of the first order differential calculus $(\Omega^{1}(\cla),d)$ with respect to the coproduct. For the rest of this subsection we concentrate on a general first order differential calculus on a Hopf algebra. 

\bdfn \cite[Definition 1.2]{Woronowicz}
Let $(\cla, \Delta)$ be a Hopf algebra and $(\Omega^1(\cla), d)$ be a first order differential calculus on it. Then
\begin{enumerate}
	\item $(\Omega^1(\cla), d)$ is said to be left covariant if for any finite set of elements $a_k$, $b_k$ in $\cla$ satisfying $\sum_k a_k d(b_k) = 0$, we have
	\[ \sum_k \Delta(a_k) (\id \otimes d)\Delta(b_k) = 0. \]
	\item $(\Omega^1(\cla), d)$ is said to be right covariant if for any finite set of elements $a_k$, $b_k$ in $\cla$ satisfying $\sum_k a_k d(b_k) = 0$, we have
	\[ \sum_k \Delta(a_k) (d \otimes \id)\Delta(b_k) = 0. \]
	\item $(\Omega^1(\cla), d)$ is said to be bicovariant if it is both left and right covariant.
\end{enumerate}
\edfn

We record the following result for a cocommutative Hopf algebra. Recall that a Hopf-algebra is said to be cocommutative if $\Delta(a)=a_{(1)}\ot a_{(2)}=a_{(2)}\ot a_{(1)}$, where we have used the usual Sweedler's notations.
\blmma \label{31jan21sm2}
A left covariant calculus $(\Omega^1(\cla), d)$ on a cocommutative Hopf algebra $\cla$ is bicovariant.
\elmma

\begin{proof}
	Since $\cla$ is a cocommutative Hopf algebra, we have that for any $a$ in $\cla$, \[\Delta(a) = a_{(1)} \otimes a_{(2)} = a_{(2)} \otimes a_{(1)}.\]
	Since $(\Omega^1(\cla), d)$ is a left covariant differential calculus, for any finite set of elements $a_k$, $b_k$ in $\cla$ satisfying $\sum_k a_k d(b_k) = 0$, we have that
	\[ \sum_k \Delta(a_k)(\id \otimes d)\Delta(b_k) = 0. \]
	Thus, using the cocommutativity of $\cla$, we get
	\[ \sum_k ({a_k}_{(1)} \otimes {a_k}_{(2)})({b_k}_{(1)} \otimes d({b_k}_{(2)})) = \sum_k ({a_k}_{(2)} \otimes {a_k}_{(1)})({b_k}_{(2)} \otimes d({b_k}_{(1)})) = 0. \]
	Recall that the map $\flip : \cla \otimes \Omega^1(\cla) \to \Omega^1(\cla) \otimes \cla$ defined on simple tensors by
	\[ a \otimes e \to e \otimes a\]
	is a well-defined $\IC$\nobreakdash-linear map. Hence, we have that
	\begin{align*} &\ \flip(\sum_k ({a_k}_{(2)} \otimes {a_k}_{(1)})({b_k}_{(2)} \otimes d({b_k}_{(1)}))) = \flip(\sum_k {a_k}_{(2)} {b_k}_{(2)} \otimes {a_k}_{(1)} d({b_k}_{(1)})) \\ = & \ \sum_k {a_k}_{(1)} d({b_k}_{(1)}) \otimes {a_k}_{(2)} {b_k}_{(2)} = 0. \end{align*}
	Thus we have proved that for any $a_k$, $b_k$ in $\cla$ satisfying $\sum_k a_k d(b_k) = 0$, \[ \sum_k \Delta(a_k)(d \otimes \id)\Delta(b_k) = 0. \]
	But this is precisely the definition of a right covariant differential calculus. Hence $(\Omega^1(\cla), d)$ is a bicovariant differential calculus as well.
\end{proof}

\subsection{Connes' space of forms over a $\ast$\nobreakdash-algebra}
\label{Connes_forms}
In this subsection, we recall the construction of Connes' space of forms in a way more suited to our purposes. We consider a triple $(\cla,\clh,\cld)$ (called a Dirac triple), where $\cla$ is a $\ast$\nobreakdash-algebra faithfully represented on $\clb(\clh)$ and $\cld$ is a priori an unbounded operator with $[\cld,a]$ in $\clb(\clh)$ for all $a$ in $\cla$. As in \cite{Cuntz}, we do not assume any summability or compactness of $(\cla,\clh,\cld)$. For a $\ast$\nobreakdash-algebra $\cla$, recall the reduced universal differential graded algebra $(\Omega^{\bullet}(\cla):=\oplus_{k} \Omega^{k}(\cla),\delta)$ from \cite{Connes}. Given a Dirac triple $(\cla,\clh,\cld)$ over $\cla$, there is a well defined $\ast$\nobreakdash-representation $\Pi$ of $\Omega^{\bullet}(\cla)$ on $\clb(\clh)$ given by the following (see \cite{Landi}):
\[
\Pi(a_{0}\delta a_{1}...\delta a_{k})=a_{0}[\cld,a_{1}]...[\cld,a_{k}], \ \text{where} \ a_{0},...,a_{k}\in\cla.
\]
Let $J^{k}_{0}=\{\omega\in\Omega^{k}(\cla):\Pi(\omega)=0\}$. The Connes' space of $k$\nobreakdash-forms is defined to be 
\[
\Omega^{k}_{\cld}(\cla)=\Pi(\Omega^{k}(\cla))/\Pi(\delta J^{k-1}_{0}).
\]
$\Pi(\delta J^{k-1}_{0})$ is a two sided ideal of $\Pi(\Omega^{k}(\cla))$ and is called the space of junk forms. For any element $\omega$ in $\Omega^{k}(\cla)$, we denote the image of $\Pi(\omega)$ in $\Omega^{k}_{\cld}(\cla)$ by $\overline{\Pi(\omega)}$. It can be shown that $(\Omega_{\cld}(\cla):=\oplus_{k}\Omega^{k}_{\cld}(\cla),d)$ where $d$ is defined as $d\overline{\Pi(\omega)}:=\overline{\Pi(\delta\omega)}$ is a differential calculus. In this paper, we shall construct tame calculi on a $\ast$\nobreakdash-algebras using the above prescription.


\section{Tame differential calculus from toral actions}
In this section, we first find a large class of $\textup{C}^{\ast}$\nobreakdash-algebras on which a tame calculus can be constructed. On such a calculus, a unique Levi-Civita connection exists for any strongly $\sigma$\nobreakdash-compatible metric. Writing the connection and curvature forms explicitly using the Christoffel symbols, we shall then go on to obtain a version of Bianchi identity in our setup. Finally, we show that toral action may be used to produce $\textup{C}^*$\nobreakdash-algebras in the class mentioned above; in this case, the Christoffel symbols take a further simple form. 
\bthm

\label{main}
Let $A$ be a ($\textup{C}^*$\nobreakdash-) algebra generated by $n$\nobreakdash-isometries $S_{1},\dots ,S_{n}$. Let us denote the dense $\ast$-subalgebra of $A$ generated by $S_{1},\ldots,S_{n}$ by $\cla$. Assume that $\cla$ admits a separating family of derivations $\partial_{1},\dots ,\partial_{n}$ in the sense that $\partial_{i}(S_{j})=\delta_{ij}S_{i}$ and $\partial_{i}(S_{j}^{\ast})=-\delta_{ij}S_{j}^{\ast}$. Then 	\begin{enumerate}
	\item There is a tame differential calculus on $\cla$ satisfying the conditions of \cite[Proposition 6.8]{conformal_bhowmick}, i.e.,  $\Omega^{1}(\cla)$ is a free $\cla$\nobreakdash-module with basis $e_{1},\ldots,e_{n}$ and for each $i=1,\dots,n$, $d(e_{i})=0$. 
	\item $\Omega^{2}(\cla)$ is a free module with basis $\{e_{i}\wedge e_{j}\}_{1\leq i<j\leq n}$, where $\wedge:\Omega^{1}(\cla)\ot_{\cla}\Omega^{1}(\cla)\raro\Omega^{2}(\cla)$ is the right $\cla$\nobreakdash-linear map as in \cite[Definition 2.1]{Bhowmick_gos_joardar}. 
	\item The operator $d : \Omega^1(\mathcal{A}) \rightarrow \Omega^2(\mathcal{A})$ is given by the following formula:
	\begin{equation}
	\label{d}
	d(\sum_{i=1}^{n}e_{i}a_{i})=\sum_{1\leq p<q\leq n}e_{p}\wedge e_{q}\Big(\sum_{i=1}^{n}(\partial_{p}(a_{i}S_{i}^{\ast})\partial_{q}(S_{i})-\partial_{q}(a_{i}S_{i}^{\ast})\partial_{p}(S_{i}))\Big)
	\end{equation}
	\item The calculus admits a bilinear metric $g_{0}$ given by $g_{0}(e_{i}\ot e_{j})=\delta_{ij}$.
\end{enumerate}
	\ethm

	\begin{proof} 
	Let $N=2^{[\frac{n}{2}]}$, where $[\cdot]$ is the greatest integer function. We fix a faithful state $\tau$ on $A$ and let $L^{2}(A,\tau)$ be the corresponding GNS space. Let $\clh$ be the Hilbert space $L^{2}(A,\tau)\ot\mathbb{C}^{N}$ and we define a representation of $\cla$ on $\clh$ by $\pi(a):=a\ot\mathbb{I}$. Let $\cld$ be the densely defined (unbounded) operator on $\clh$ given by $\sum_{i=1}^{n}\partial_{i}\ot\gamma_{i}$, where $\{\gamma_{i}\}_{i=1}^{n}$ are standard gamma matrices acting on the vector space $\mathbb{C}^{N}$ satisfying $\gamma_{i}^{2}=\mathbb{I}$ and $\gamma_{i}\gamma_{j}=-\gamma_{j}\gamma_{i}$ for $i\neq j$. It is easy to see that for each $a$ in $\cla$, $[\cld,\pi(a)]=\sum_{i=1}^{n}\partial_{i}(a)\otimes \gamma_{i} \in \clb(\clh)$ and hence $(\cla,\clh,\cld)$ is a spectral data in our sense.
    \subsection*{Connes space of one\nobreakdash-forms}
    Recall the definition of Connes space of $k$\nobreakdash-forms $\Omega^{k}_{\cld}(\cla)$ from Subsection \ref{Connes_forms}. We claim that $\Omega^{1}_{\cld}$ is a free bimodule of rank $n$ with basis $\{1\ot\gamma_{i}\}_{i=1}^{n}$. To see this, we observe that $\pi(a)[\cld,\pi(b)]=\sum_{i}a\partial_{i}(b)\ot\gamma_{i}$ which yields $\Omega^{1}_{\cld}(\cla)\subset \cla\oplus\ldots\oplus\cla$. For the other inclusion, it is enough to note that $\pi(S_{i}^{\ast})[\cld,\pi(S_{i})]=S_{i}^{\ast}S_{i}\ot\gamma_{i}=1\otimes \gamma_i$, since $S_{i}$ is an isometry. We henceforth write $e_i$ for $1\otimes \gamma_i$, $i=1,\dots,n$.
    \subsection*{Connes space of two\nobreakdash-forms}
	By definition, $\Omega^{2}_{\cld}(\cla)=\Pi(\Omega^{2})/\Pi(\delta J_{0}^{1})$. It can be proved exactly along the lines of \cite{Cuntz} that for $r=\frac{n(n-1)}{2}+1$,
		\[
		\Pi(\Omega^{2}(\cla))=\cla\underbrace{\oplus\ldots\oplus}_{r\text{-times}}\cla.
		\]
		The space of junk forms is given by the free module $\pi(\cla)\ot\mathbb{I}$ in $\clb(\clh)$ proving that $\Omega^{2}_{\cld}(\cla)$ is a free module of rank $\frac{n(n-1)}{2}$ with basis $\{1\ot\gamma_{i}\gamma_{j}\}_{1\leq i<j\leq n}$. The fact that the junk forms are contained in $\pi(\cla)\ot\mathbb{I}$ follows from similar arguments as in \cite{Cuntz}. For the equality, we let $\omega=S_{1}^{\ast}\delta{S_{1}}+\delta(S_{1}^{\ast})S_{1}$ and observe that $\Pi(\omega)=0$, $\Pi(\delta(\omega))=2\ot\mathbb{I}$, which proves the desired result. As above, we write $e_{ij}$ for $1\otimes \gamma_i\gamma_j$, $1\leq i < j \leq n$.
		\subsection*{The multiplication map $\wedge$}
		The multiplication map $\wedge:\Omega^{1}_{\cld}(\cla)\ot_{\cla}\Omega^{1}_{\cld}(\cla)\raro \Omega^{2}_{\cld}(\cla)$ is surjective and is given by
		\begin{equation}
		\label{mult}
		\wedge(\sum_{i=1}^{n}e_{i}a_{i},\sum_{i=1}^{n}e_{i}b_{i})=\sum_{i<j}e_{ij}(a_{i}b_{j}-a_{j}b_{i}).
		\end{equation}
		Proof of the formula \eqref{mult} follows exactly the same lines of Lemma 3.3 of \cite{Cuntz}. Now the existence of the canonical bilinear metric and the fact that $(\Omega_{\cld}(\cla),d)$ is a tame calculus can be established by the same considerations as in \cite{Cuntz}. For the proof of the formula \eqref{d}, see Lemma 5.5 of \cite{Satyajit_jncg}. The fact that $d(e_{i})=0$ is a consequence of the formula \eqref{d}. 
		\end{proof}
	\brmrks
	\label{wedge_extension}
	\hfill
	\begin{enumerate}
		\item We shall not calculate Connes' space of $n$\nobreakdash-forms for $n\geq 3$ as we shall not need them explicitly in this paper. However, to prove the Bianchi identity, we shall use the fact that Connes' space of forms is a differential graded algebra (see Definition 17 of \cite{Landi}) i.e., the wedge product $\wedge$ and the differential $d$ extends naturally to the spaces of $n$\nobreakdash-forms for all $n$. 
		\item We shall say a canonical tame calculus on $A$ to refer to the calculus obtained on the dense $\ast$\nobreakdash-subalgebra $\cla$.
		\item By a smooth deformation parameter $k\in A$, we shall mean that $k$ is invertible and $k,k^{-1}\in\cla$.  
	\end{enumerate}
	\ermrks

	\textbf{Henceforth, whenever we talk about a canonical differential calculus on a $\textup{C}^*$\nobreakdash-algebra $A$ we mean the differential calculus obtained in Theorem \ref{main} on $\cla$ unless mentioned otherwise.} 
	
	\medskip \noindent
	Since the differential calculus is tame, for any strongly $\sigma$\nobreakdash-compatible metric $g$, there exists a unique Levi-Civita connection to be denoted by $\nabla_{g}$. Let us denote the Christoffel symbols for a Levi-Civita connection $\nabla_{g}$ by $\Gamma^{i}_{jk}$ i.e.
	\[
	\nabla_{g}(e_{i})=\sum_{j,k=1}^{n}e_{j}\ot e_{k}\Gamma^{i}_{jk}.
	\]
	In particular, we define the connection  1-forms $(\omega_{ij})_{i,j=1,\ldots,n}$ by \[\omega_{ij}=\sum_{k=1}^{n}e_{k}\Gamma^{i}_{jk} \quad \text{so that} \quad \nabla_{g}(e_{i})=\sum_{j=1}^{n}e_{j}\ot\omega_{ij}.\] To define the curvature two\nobreakdash-forms $(\Omega_{ij})_{i,j=1,\ldots,n}$, note that the differential calculus obtained as in Theorem \ref{main} satisfies the conditions of Proposition 6.8 of \cite{conformal_bhowmick} and hence the results obtained in \cite{conformal_bhowmick} continue to hold here. Therefore we have the curvature operator:
	\begin{equation}
	\label{curvature}
	R(\nabla_{g})(e_{i})=\sum_{j,k,l=1}^{n}e_{j}\ot e_{k}\ot e_{l} r^{i}_{jkl},
	\end{equation}
	where \[r^{i}_{jkl}=\frac{1}{2}\sum_{p=1}^{n}\Big((\Gamma^{p}_{jk}\Gamma^{i}_{pl}-\Gamma^{p}_{jl}\Gamma^{i}_{pk})-\partial_{l}(\Gamma^{i}_{jk})+\partial_{k}(\Gamma^{i}_{jl})\Big).\] Using the facts that $\wedge$ is right $\cla$\nobreakdash-linear and $e_{i}\wedge e_{j}=-e_{j}\wedge e_{i}$ for all $i,j=1,\ldots,n$, we obtain the curvature two\nobreakdash-forms $\Omega_{ij}$ as
	\begin{equation}
	\label{curvatureform}
	\Omega_{ij}=\sum_{1\leq k<l\leq n}\Big(\sum_{p=1}^{n}[(\Gamma^{p}_{jk}\Gamma^{i}_{pl}-\Gamma^{p}_{jl}\Gamma^{i}_{pk})-\partial_{l}(\Gamma^{i}_{jk})+\partial_{k}(\Gamma^{i}_{jl})]\Big)e_{k}\wedge e_{l}.
	\end{equation}
Having defined the connection and curvature forms, we shall prove a relation between them which will be used to prove the Bianchi identity.

\bppsn
\label{Bianchiidentity1}
Let $g$ be a strongly $\sigma$\nobreakdash-compatible metric and $\nabla_{g}$ be the associated Levi-Civita connection on the tame calculus obtained in Theorem \ref{main}. Let $\omega$ and $\Omega$ be the matrices of connection and curvature forms, respectively. Then the following identity holds:
\begin{equation}\label{alg_Bianchiidentity}
	\Omega_{ij}=d\omega_{ij}+\sum_{p=1}^{n}\omega_{pj}\wedge\omega_{ip},
\end{equation} for $i,j=1,\dots,n$.
	\eppsn
	\begin{proof}
		Applying $d$ on both sides of $\omega_{ij}=\sum_{p=1}^{n}\Gamma^{i}_{jp}e_{p}$, together with the formula \eqref{d} and the fact that the derivations are separating in the sense of Theorem \ref{main}, we obtain
		\begin{equation}\label{1}
		\begin{aligned}
		d\omega_{ij}&=\sum_{1\leq k<l\leq n}e_{k}\wedge e_{l}\Bigg(\sum_{p=1}^{n}\Big(\partial_{k}(\Gamma^{i}_{jp}S_{p}^{\ast})\partial_{l}(S_{p})-\partial_{l}(\Gamma^{i}_{jp}S_{p}^{\ast})\partial_{k}(S_{p})\Big)\Bigg)\\
		&=\sum_{1\leq k<l\leq n}e_{k}\wedge e_{l}\Big(\partial_{k}(\Gamma^{i}_{jl})-\partial_{l}(\Gamma^{i}_{jk})\Big)
	\end{aligned}
\end{equation}
		Now $e_{i}$'s are central, $\wedge$ is right $\cla$\nobreakdash-linear and $e_{k}\wedge e_{l}=-e_{l}\wedge e_{k}$ for all $k,l=1,\ldots,n$, so that we have
		\begin{equation}
		\label{2} 
		\begin{aligned}
		\sum_{p=1}^{n}\omega_{pj}\wedge\omega_{ip}&=(\sum_{k=1}^{n}e_{k}\Gamma^{p}_{jk})\wedge (\sum_{l=1}^{n}e_{l}\Gamma^{i}_{pl})\\
		&=\sum_{1\leq k<l\leq n}e_{k}\wedge e_{l}\Big(\sum_{p=1}^{n}(\Gamma^{p}_{jk}\Gamma^{i}_{pl}-\Gamma^{p}_{jl}\Gamma^{i}_{pk})\Big)
		\end{aligned}
		\end{equation}
		Combining (\ref{1}) and (\ref{2}), we obtain (\ref{alg_Bianchiidentity}).
		\end{proof} 
\bcrlre[The Bianchi identity] The Bianchi identity holds in the following form
\begin{eqnarray}
\label{Bianchiidentity}
d\Omega_{ij}=\sum_{p=1}^{n}(\Omega_{pj}\wedge\omega_{ip}-\omega_{pj}\wedge\Omega_{ip}).
\end{eqnarray}
\ecrlre
\begin{proof} As mentioned in Remark \ref{wedge_extension}, $\wedge$ and $d$ extend to the space of all $n$\nobreakdash-forms. Applying the exterior derivative $d: \Omega^2(\A) \to \Omega^3(\A)$ to both sides of the equation \eqref{alg_Bianchiidentity}, we get
	\begin{eqnarray*}
	d \Omega_{ij} &=& \sum_{p=1}^n d \omega_{pj} \wedge \omega_{ip} - \sum_{p=1}^n \omega_{pj} \wedge d \omega_{ip}\\
	&=& \sum_{p=1}^n (\Omega_{pj} - \sum_{q=1}^n \omega_{qj} \wedge \omega_{pq}) \wedge \omega_{ip} - \sum_{p=1}^n \omega_{pj} \wedge (\Omega_{ip} - \sum_{r=1} \omega_{rp} \wedge \omega_{ir}) \ (by \ \eqref{alg_Bianchiidentity})\\
	&=& \sum_{p=1}^n \Omega_{pj} \wedge \omega_{ip} - \sum_{p,q = 1}^n \omega_{qj} \wedge \omega_{pq} \wedge \omega_{ip}\\
	&& \ - \sum_{p=1}^n \omega_{pj} \wedge \Omega_{ip} + \sum_{p,r = 1}^n \omega_{pj} \wedge \omega_{rp} \wedge \omega_{ir}\\
	&=& \sum_{p=1}^n (\Omega_{pj} \wedge \omega_{ip} - \omega_{pj} \wedge \Omega_{ip}),
	\end{eqnarray*}which was to be obtained.
	\end{proof}	\brmrk
	Note the slight difference between \eqref{Bianchiidentity} and the Bianchi identity from classical differential geometry, (in matrix notation, $d\Omega=\omega \wedge \Omega - \Omega \wedge \omega$, see page 297, Eq.(2) of \cite{milnor}) due to the noncommutativity of the Christoffel symbols.
	\ermrk 

	One of the ways to generate a family of derivations as in Theorem \ref{main} is to look for toral actions.
	\blmma
	\label{toralaction}
	Let $\cla$ be a finitely generated ($\textup{C}^*$\nobreakdash-) algebra with generators $S_{1},\dots,S_{n}$, admitting an action of $\mathbb{T}^{n}$ such that on the generators, the action is given by
	\[
	(\lambda_{1},\ldots,\lambda_{n})(S_{j}):=\lambda_{j}S_{j}, \ \lambda_{j}\in\mathbb{T}.
	\]
	Then there exist $n$ derivations which are separating in the sense of Proposition \ref{main}\elmma
	\begin{proof}
		The proof is immediate.
	\end{proof}

We shall now determine the Christoffel symbols for the unique Levi-Civita connection $\nabla_{g}$ for a strongly $\sigma$\nobreakdash-compatible metric $g$ on the differential calculus obtained in Theorem \ref{main}. 


\bcrlre
\label{main_christoffel}
Let $(\Omega(\cla),d)$ be the tame differential calculus as obtained in Theorem \ref{main}. Then for any strongly $\sigma$\nobreakdash-compatible metric $g$, we have the following formula for the Christoffel symbols:
\begin{equation}
\label{Christoffel}
\begin{aligned}
\Gamma^p_{ml} = \frac{1}{2}\Big( \sum_j g^{lj} \partial_m (g_{pj}) + \sum_i g^{mi} \partial_l (g_{ip}) -\sum_{ijn} g_{pn} g^{li} g^{mj} \partial_n (g_{ij}) \Big)
\end{aligned} 
\end{equation}
\ecrlre

\begin{proof}
It is clear that the calculus satisfies the assumptions of Theorem \ref{christoffelgeneral}. Further, $d(e_{i})=0$ for all $i$, which implies that we can choose a torsionless connection $\nabla_{0}$ such that $\nabla_{0}(e_{i})=0$ $\forall i$. Therefore the formula is obtained by putting $(\Gamma_{0})^{i}_{jk}=0$ for all $i,j,k$ in equation (\ref{10feb21sm15}).
\end{proof}

\begin{crlre}
	\label{christoffelreduced}
	In addition to the hypothesis of Corollary \ref{main_christoffel}, assume further that the algebra $\A_G$ as in Remark \ref{metricalgbera} is closed under $\partial_n$ for all $n$. Then the Christoffel symbols $\Gamma^p_{ml}$ have the following further reduced form:
	\begin{equation} \label{19feb21sm2}
	\Gamma^p_{ml} = - \frac{1}{2} \sum_ n g_{pn} \Big( \partial_m(g^{ln}) + \partial_l (g^{mn}) - \partial_n (g^{ml}) \Big).
	\end{equation}
\end{crlre}
\begin{proof}
	Note that $g_{ij}$, $g^{kl}$, $\partial_m(g_{pq})$ and $\partial_n(g^{rs})$ all commute with each other. Then using the Leibniz relation of $\partial_m$, we compute the following:
	\allowdisplaybreaks{
	\begin{align*}
	\Gamma^p_{ml} = \ & \frac{1}{2}\Big( \sum_j g^{lj} \partial_m (g_{pj}) + \sum_i g^{mi} \partial_l (g_{ip}) - \sum_{ijn} g_{pn} g^{li} g^{mj} \partial_n (g_{ij}) \Big) \\
	& ({\rm using} \ \eqref{Christoffel})\\
	= \ & \frac{1}{2}\Big( \sum_j (\partial_m (g^{lj} g_{pj}) - \partial_m (g^{lj}) g_{pj} ) \\
	& + \sum_i (\partial_l (g^{mi} g_{ip}) - \partial_l (g^{mi}) g_{ip} ) \\
	& - \sum_{ijn} g_{pn} g^{li} (\partial_n (g^{mj} g_{ij}) - \partial_n (g^{mj}) g_{ij}) \Big) \\
	= \ & \frac{1}{2} \Big( - \sum_j \partial_m(g^{lj}) g_{pj} - \sum_i \partial_l (g^{mi}) g_{ip} 
	+ \sum_{ijn} g_{pn} g^{li} g_{ij} \partial_n (g^{mj}) \Big) \\
	= \ & - \frac{1}{2} \sum_ ng_{pn} \Big( \partial_m(g^{ln}) + \partial_l (g^{mn}) - \partial_n (g^{ml}) \Big),
	\end{align*}
}which was to be obtained.
\end{proof}

\brmrk
Under the hypothesis of the Corollary \ref{christoffelreduced}, we see from (\ref{19feb21sm2}) that the Christoffel symbols commute among themselves. Hence the traditional form of the Bianchi identity holds.
\ermrk

\section{Examples: noncommutative torus,\\ Group algebra of the free group, Cuntz algebra} 
In this section, we discuss some explicit examples satisfying the hypothesis of Lemma \ref{toralaction}. It turns out that on the group algebra $\mathbb{C}[\mathbb{F}_n]$, the tame calculus is also bicovariant.

\subsection{The noncommutative $n$\nobreakdash-torus} 
We fix a skew-symmetric $n\times n$ matrix $\Theta=((\theta_{kl}))$. The noncommutative $n$\nobreakdash-torus $\textup{C}(\mathbb{T}^{n}_{\Theta})$ is the universal ($\textup{C}^*$\nobreakdash-) algebra generated by $n$\nobreakdash-unitaries $U_{1},U_{2},\ldots,U_{n}$ subject to the following commutation relations:
\[
U_{k}U_{l}=e^{2i\pi \theta_{kl}}U_{l}U_{k}, \quad k,l=1,\dots,n.
\]
By the universal property, there is a $\mathbb{T}^{n}$\nobreakdash-action on $\textup{C}(\mathbb{T}^{n}_{\Theta})$ as in Lemma \ref{toralaction} and consequently there is a canonical tame calculus as obtained in Theorem \ref{main}.

\begin{rmrk}
	The unique Levi-Civita connection $\nabla_g$ compatible with a strongly $\sigma$\nobreakdash-compatible metric $g$ on the canonical tame calculus of $\textup{C}(\IT^n_\Theta)$, obtained by virtue of Theorem \ref{13jul21sm5}, is the same as that obtained in \cite{rosenberg}. A detailed comparison result can be found in \cite{conformal_bhowmick}.
\end{rmrk}

\subsection{The group algebra of the free group} Let $\mathbb{F}_n$ be the free group on $n$ generators (say, $g_1,\dots,g_n$) and $\textup{C}^{\ast}_{r}(\mathbb{F}_n)$ be the reduced group $\textup{C}^{\ast}$-algebra. Then there is a $\mathbb{T}^n$\nobreakdash-action on $\textup{C}^{\ast}_{r}(\mathbb{F}_{n})$ as in Lemma \ref{toralaction} and we obtain a canonical tame calculus. Note that for the $\textup{C}^{\ast}$-algebra $A=\textup{C}^{\ast}_{r}(\mathbb{F}_{n})$, the dense $\ast$\nobreakdash-subalgebra $\cla$ is nothing but the group algebra $\mathbb{C}[\mathbb{F}_{n}]$.
\bthm \label{13jul21sm4}
The canonical tame calculus on $\mathbb{C}[\mathbb{F}_n]$ is a bicovariant differential calculus. The canonical basis elements $e_i$  are left (and right) invariant.
\ethm
\begin{proof}
	In order to prove that the calculus is left covariant, we need to show that if $\sum_k a_k d(b_k) = 0$ for some $a_k$, $b_k$ in $\cla$, then $\sum_k \Delta(a_k) (\id \otimes d) \Delta(b_k) = 0$. Let us observe that since $d(a) = \sum_{i=1}^n \partial_{i}(a) \otimes \gamma_i$, $\sum_k a_k d(b_k) = 0$ implies that for all $i$,
	\begin{equation} \label{10feb21sm2} \sum_k a_k \partial_i (b_k) = 0. \end{equation}
	Since our algebra is a cocommutative Hopf algebra, there exist $c_{k,j}$, $c^\prime_{k,l}$ in $\IC$ and group-like elements $a_{k,j}$, $b_{k,l}$ in $\cla$ such that for all $k$
	\begin{equation} \label{31jan21sm1}
	a_k = \sum_j c_{k,j} a_{k,j} \ \textnormal{and} \ b_k = \sum_l c^\prime_{k,l} b_{k,l}.
	\end{equation}
	Let the element $b_{k,l}$ be presented by the reduced string $g_{k,l,1} g_{k,l,2} \dots g_{k,l,{n_{k,l}}}$, where each $g_{k,l,m}$ is from the set of generators of the Hopf algebra. Since the Hopf algebra is freely generated and each $b_{k,l}$ is group-like, such a presentation exists and is unique. We denote by $p_{k,l,i}$ the sum of occurences of the generator $g_i$ in the string, counting each occurence of $g_i^{-1}$ as $-1$ and $g_i$ as $+1$. We note that $\partial_j(g_i) = \delta_{ji} g_i$. Then, by the Leibniz rule, we have that $\partial_i(b_{k,l}) = p_{k,l,i} b_{k,l}$. Hence from \eqref{10feb21sm2} we get
	\begin{equation} \label{10feb21sm1}
	\sum_{k,j,l} c_{k,j} c^\prime_{k,l} a_{k,j} \partial_i(b_{k,l}) = \sum_{k,j,l} c_{k,j} c^\prime_{k,l} p_{k,l,i} a_{k,j} b_{k,l} = 0.
	\end{equation}
	Hence, we have that for all $i$,
	\begin{equation} \label{31jan21sm3}
	\begin{aligned}
	&\sum_k \Delta(a_k) (\id \otimes \partial_i) \Delta(b_k) = \sum_{k,j,l} c_{k,j} c^\prime_{k,l} \Delta(a_{k,j}) (\id \otimes \partial_i) \Delta(b_{k,l}) \\
	= \ &\sum_{k,j,l} c_{k,j} c^\prime_{k,l} (a_{k,j} \otimes a_{k,j}) (b_{k,l} \otimes p_{k,l,i} b_{k,l}) \\
	= \ &\sum_k c_{k,j} c^\prime_{k,l} p_{k,l,i} a_{k,j} b_{k,l} \otimes a_{k,j} b_{k,l} = \sum_{k,j,l} c_{k,j} c^\prime_{k,l} p_{k,j,l} \Delta(a_{k,j} b_{k,l}) \\
	= \ & \Delta(\sum_{k,j,l} c_{k,j} c^\prime_{k,l} p_{k,j,l} a_{k,j} b_{k,l}) \\ 
	= \ & 0 \ \textnormal{(by \ \eqref{10feb21sm1})}.
	\end{aligned}
	\end{equation}Then we have that
	\begin{equation*}
		\begin{aligned}
		&\sum_k \Delta(a_k) (\id \otimes d)\Delta(b_k) 
		&=\sum_i \big(\sum_k \Delta(a_k) (\id \otimes \partial_i)\Delta(b_k)\big) \otimes \gamma_i = 0 \ \rm{(by \ \eqref{31jan21sm3})}.
		\end{aligned}
	\end{equation*}
	This proves that the canonical tame calculus is left covariant. By Lemma \ref{31jan21sm2}, we know that this implies the calculus is bicovariant.\\
	To see that the canonical basis elements $e_i$ are left invariant, note that each $e_i = \delta_{g_i}^\ast d(\delta_{g_i})$, where $g_i$ are the corresponding generating elements of the group. Then we have that
	\begin{equation*}
		\begin{aligned}
		&\Delta_{\Omega^1_\cld(\cla)} (e_i) = \Delta_{\Omega^1_\cld(\cla)} (\delta_{g_i}^\ast d(\delta_{g_i}))
		:= \Delta(\delta_{g_i}^\ast) (\id \otimes d)\Delta(\delta_{g_i})\\
		&=(\delta_{g_i}^\ast \otimes \delta_{g_i}^\ast) (\delta_{g_i} \otimes d(\delta_{g_i}))
		= 1 \otimes \delta_{g_i}^\ast d(\delta_{g_i}) = 1 \otimes e_i.
		\end{aligned}
	\end{equation*}
	The proof of right invariance of $e_i$ is similar.
\end{proof}

\subsection{The Cuntz algebra $\clo_{n}$} The Cuntz algebra $\clo_{n}$ is the universal ($\textup{C}^*$\nobreakdash-) algebra generated by $n$\nobreakdash-isometries $S_{1},\ldots,S_{n}$ such that $\sum_{i=1}^{n}S_{i}S_{i}^{\ast}=1$. The universal property produces a $\mathbb{T}^{n}$\nobreakdash-action on $\clo_{n}$ as in Lemma \ref{toralaction}. Consequently there is a canonical tame calculus on $\clo_{n}$ by Theorem \ref{main}. Note that the derivations obtained here are separating in the sense of Theorem \ref{main}. In \cite{Cuntz}, for the Cuntz algebra with three generators, instead of the action of $\mathbb{T}^3$, a canonical action of $\textup{SO}(3)$ is considered. The space of one-forms thus constructed is a centered bimodule of rank three and hence isomorphic as bimodules to the space of one-forms considered in the present paper. However, the two exterior derivatives are different, leading to non-equivalent calculi. In particular, the derivations are not seperating. Moreover, for the present calculus on $\mathcal{O}_n$, thanks to \eqref{Christoffel}, the Christoffel symbols of the Levi-Civita connection corresponding to the bilinear metric $g_0$ vanish and consequently the Ricci tensor as well as the scalar curvature vanish (see Proposition 6.8 of \cite{conformal_bhowmick}). But for the calculus constructed in \cite{Cuntz}, the corresponding metric $g_0$ on the isomorphic bimodule of one-forms leads to the value $-\frac{3}{4}$ for the scalar curvatue of the Levi-Civita connection obtained there (see Theorem 4.7 of \cite{Cuntz}).

\section{Towards a Gauss-Bonnet theorem on the canonical calculus of rank two}

In this section, we shall prove a version of the Gauss-Bonnet theorem for the calculus of rank two obtained here. The formulation is in the spirit of Connes, Khalkhali, Ponge et al. (see \cite{khalkhali2, Sitarz, tretkof}). Since the tame calculus obtained in this paper admits a canonical bilinear metric $g_{0}$, it makes sense to consider the widely studied conformally deformed metric $kg_{0}$ for some smooth deformation parameter $k$. As the Gauss-Bonnet theorem considered so far in the literature deals mostly with a conformally deformed metric, let us briefly discuss about such a metric first. Note that such metrics are strongly $\sigma$\nobreakdash-compatible in our sense.

\subsection{Conformally deformed metrics} The tame differential calculus obtained in Theorem \ref{main} satisfies the assumptions of the Corollary \ref{christoffelgeneral} as well as the Proposition 5.10 of \cite{conformal_bhowmick}. Hence we have the following lemma whose proof is immediate. Recall from point (3) of Remark \ref{wedge_extension} that $k$ being a smooth deformation parameter in a $\textup{C}^{\ast}$-algebra $A$ means that $k$ is invertible and $k,k^{-1}\in\cla$.

\blmma
\label{formulae}
For some smooth deformation parameter $k$ in a $\textup{C}^{\ast}$-algebra $A$ let $g=kg_{0}$ be a conformally deformed metric on the differential calculus obtained in Theorem \ref{main}. Then the Christoffel symbols, for the Levi-Civita connection $\nabla_{g}$ are given by:
\begin{equation}
\label{Christoffelconformal}
\Gamma^{i}_{jl}=\frac{1}{2}\Big(\delta_{il}k^{-1}\partial_{j}(k)+\delta_{ij}k^{-1}\partial_{l}(k)-\delta_{jl}k^{-1}\partial_{i}(k)\Big).
\end{equation} 
\elmma

From now on, we consider a $\textup{C}^*$\nobreakdash-algebra $A$ generated by two isometries, say $S_{1}, S_{2}$ such that the dense $\ast$-subalgebra $\cla$ generated by $S_{1},S_{2}$ admits two separating derivations $\partial_{1},\partial_{2}$ in the sense of discussed before Theorem \ref{main}. Then there is a tame differential calculus as obtained in Theorem \ref{main}, which we call the canonical calculus of rank two on $A$. The space of 2-forms $\Omega^{2}(\cla)$ is a rank one bimodule with basis $e_{1}\wedge e_{2}$, hence for any $\theta$ in $\Omega^{2}(\cla)$, there is a unique $a$ in $\cla$ such that $\theta= e_{1}\wedge e_{2}a$, which amounts to choosing an orientation. Fixing a state $\tau$, we define the integral of $\theta$ as 
\[
\int_{\tau}\theta:=\tau(a).
\]
We briefly recall the definition of the scalar curvature of a torsionless connection $\nabla$ on a tame calculus with a strongly compatible metric $g$. This is obtained by contracting the Ricci curvature. In the following, ${\rm ev}:\Omega^{1}(\cla)^{\ast}\ot_{\cla}\Omega^{1}(\cla)\raro\cla$ is the right $\cla$-linear map sending $\omega^{\ast}\ot\eta$ to $\omega^{\ast}(\eta)$. For the notation $\rho$, consult Section 2.1 of \cite{Cuntz}.
\bdfn [Definition 6.6 of   \cite{conformal_bhowmick}]
\label{Scal} For a tame differential calculus $(\Omega^{1}(\cla),\wedge,d)$ and a torsionless connection $\nabla$, the Ricci curvature ${\rm Ric}$ is defined as the element in $\Omega^{1}(\cla)\ot_{\cla}\Omega^{1}(\cla)$ given by \begin{displaymath}
{\rm Ric}:=({\rm id}\ot_{\cla}{\rm ev}\circ\rho)(\Theta),
\end{displaymath}
where $\Theta$ is the curvature operator. The scalar curvature ${\rm Scal}$  corresponding to a strongly $\sigma$-compatible metric $g$ is defined by
\begin{displaymath}
{\rm Scal}={\rm ev}(V_{g}\ot_{\cla}{\rm id})({\rm Ric})\in\cla.
\end{displaymath}
\edfn 
We have the following formula of the scalar curvature for the conformally deformed metric $g$ whose proof is similar to that of Theorem 6.12 of \cite{conformal_bhowmick}.
\begin{equation}
\label{scalarconformal}
{\rm Scal}=-(\partial_{1}^{2}+\partial_{2}^{2})(k)-k\partial_{1}(k^{-1})\partial_{1}(k)-k\partial_{2}(k^{-1})\partial_{2}(k).
\end{equation}
Using the Leibniz rule for $\partial_{i}$'s we have the following alternative form of the scalar curvature which will be used subsequently
\begin{equation}
\label{scalarconformal1}
{\rm Scal}=-k\partial_{1}\Big(k^{-1}\partial_{1}(k)\Big)-k\partial_{2}\Big(k^{-1}\partial_{2}(k)\Big).
\end{equation}
\brmrk
We remark that as already observed in \cite{conformal_bhowmick}, the above expression for the scalar curvature coincides with the expression in the case of noncommutative $2$\nobreakdash-torus obtained in \cite{rosenberg} (Eq.(4.3) of \cite{rosenberg}).  
\ermrk 
It is well known that for a surface, the Gaussian curvature is half of the scalar curvature. Motivated by this, we define the Gaussian curvature $K$ by the following:
\[
K=\frac{1}{2}{\rm Scal}.
\] For a conformally deformed metric $kg_{0}$, we define the right and left surface integral of $K$ with respect to the state $\tau$ by

\begin{equation}
\int^{R}_{\tau} K \ dS:=\tau(K|k^{-1}|), \ \int^{L}_{\tau}K \ dS:=\tau(|k^{-1}|K)
\end{equation}

\brmrk
The analogue for the Riemannian volume form is the element ${e_{1}\wedge e_{2}|k^{-1}|}$ in $ \Omega^{2}(\cla)$, hence the above two choices for defining the surface integral of a ``function'' on the ``surface'' $\cla$. When $\tau$ is a trace, the integrals coincide and it makes sense to talk about ``the'' surface integral.
\ermrk

  From now on we assume that the $\textup{C}^*$\nobreakdash-algebra admits a tracial state $\tau$. Recall the curvature $2$\nobreakdash-form $\Omega_{ij}$ (see \eqref{curvatureform}), given by
  \begin{equation*}
	\label{curvatureform}
	\Omega_{ij}=\sum_{1\leq k<l\leq n}\Big(\sum_{p=1}^{n}[(\Gamma^{p}_{jk}\Gamma^{i}_{pl}-\Gamma^{p}_{jl}\Gamma^{i}_{pk})-\partial_{l}(\Gamma^{i}_{jk})+\partial_{k}(\Gamma^{i}_{jl})]\Big)e_{k}\wedge e_{l},
	\end{equation*}
	corresponding to the Levi-Civita connection $\nabla_{g}$ for a strongly $\sigma$-compatible metric $g$.

 \bppsn\label{GaussBonneteasy}
 Let $\tau$ be a trace on $A$ and $g = k g_0$ be a conformally deformed metric with a smooth deformation parameter $k$ on the canonical calculus of rank two. Then  $\Omega_{12}=-\Omega_{21}$ and
 \begin{equation}
 \label{GBeasy}
 \int_{\tau}\Omega_{12}=\int_{\tau}K \ dS.
 \end{equation}
 \eppsn
 \begin{proof}
Using \eqref{curvatureform} and the expression for the scalar curvature for a conformally deformed metric, a simple but tedious calculation yields:
 	\begin{equation*}
		\begin{aligned}
 		\Omega_{12}&=e_{1}\wedge e_{2}\Big(\sum_{p=1}^{2}(\Gamma^{p}_{21}\Gamma^{1}_{p2}-\Gamma^{p}_{22}\Gamma^{1}_{p1})+\partial_{1}(\Gamma^{1}_{22})-\partial_{2}(\Gamma^{1}_{21})\Big)\\
 		&=e_{1}\wedge e_{2}\Big(\Gamma^{1}_{21}\Gamma^{1}_{12}-\Gamma^{1}_{22}\Gamma^{1}_{11}+\Gamma^{2}_{21}\Gamma^{1}_{22}-\Gamma^{2}_{22}\Gamma^{1}_{21}+\partial_{1}(\Gamma^{1}_{22})-\partial_{2}(\Gamma^{1}_{21})\Big)\\
 		&=e_{1}\wedge e_{2}\Big(\frac{1}{4}k^{-1}\partial_{2}(k)k^{-1}\partial_{2}(k)+\frac{1}{4}k^{-1}\partial_{1}(k)k^{-1}\partial_{1}(k)-\frac{1}{4}k^{-1}\partial_{2}(k)k^{-1}\partial_{2}(k)\\
 		&-\frac{1}{4}k^{-1}\partial_{2}(k)k^{-1}\partial_{2}(k)
 		 -\frac{1}{2}\partial_{1}\Big(k^{-1}\partial_{1}(k)\Big)-\frac{1}{2}\partial_{2}\Big(k^{-1}\partial_{2}(k)\Big)\Big)\\
 		&=e_{1}\wedge e_{2}\Big(\frac{1}{2}(k^{-1}{\rm Scal})\Big).
		\end{aligned}
 		\end{equation*}

 	 	By a similar computation, we see that \[\Omega_{21}=-e_{1}\wedge e_{2}\Big(\frac{1}{2}(k^{-1}{\rm Scal})\Big),\] which yields the first conclusion. The second conclusion, namely, (\ref{GBeasy}), follows from the definitions of the integrals on the two sides and the fact that $|k^{-1}|=k^{-1}$ for positive $k$.
 	\end{proof}
  \brmrk
 Thanks to the equation (\ref{GBeasy}), we can call $\Omega_{12}$ the Gauss-Bonnet $2$\nobreakdash-form corresponding to a conformally deformed metric. We shall denote the Gauss-Bonnet $2$\nobreakdash-form corresponding to a smooth deformation parameter $k$ by $\Omega_{12}^{k}$.
 \ermrk 

	
	
		
 \subsection{The Gauss-Bonnet theorem} 
 
Let us consider the canonical calculus of rank two on a $\textup{C}^*$\nobreakdash-algebra $A$ with a tracial state $\tau$. For the next definition recall the notations from the previous subsection. Let $k\in A$ be a smooth deformation parameter i.e. $k$ is an invertible element and $k, k^{-1}\in\cla$. We denote the Gauss-Bonnet $2$\nobreakdash-form corresponding to the metric $kg_{0}$ by $\Omega_{12}^{k}$.
\bdfn
Let $A$ be a $\textup{C}^*$\nobreakdash-algebra admitting the canonical calculus of rank two $(\Omega^{1}(\cla),d)$. Then the calculus of rank two is said to satisfy the Gauss-Bonnet theorem if $\int_{\tau}\Omega_{12}^{k}$ is independent of the smooth deformation parameter $k$.
\edfn 
\bppsn \label{13jul21sm1}
If the canonical calculus of rank two on a $\textup{C}^*$\nobreakdash-algebra $A$ with a tracial state $\tau$ satisfies the Gauss-Bonnet theorem, then $\int_{\tau}\Omega_{12}^{k}$ is $0$ for all smooth deformation parameters $k$.
\eppsn 
\begin{proof} Let $k_{1}, k_{2}$ be two smooth deformation parameters. Then we prove that
	\begin{displaymath}
	\int_{\tau}\Omega_{12}^{k_{1}k_{2}}= \int_{\tau}\Omega_{12}^{k_{1}}+\int_{\tau}\Omega_{12}^{k_{2}},
	\end{displaymath}
	which will force the constant $\int_{\tau}\Omega_{12}^{k}$ to be zero for all smooth deformation parameters $k$. By the definition and formula (\ref{scalarconformal1}) for the scalar curvature, we have:
	\[
	\int_{\tau}\Omega^{k_{1}k_{2}}_{12}=-\tau(\partial_{1}\Big(k_{2}^{-1}k_{1}^{-1}\partial_{1}(k_{1}k_{2})\Big)+\partial_{2}\Big(k_{2}^{-1}k_{1}^{-1}\partial_{2}(k_{1}k_{2})\Big)).
	\]
	Using the Leibniz rule of $\partial_{1}$ and the traciality of $\tau$, we get
	\begin{equation*}
		\begin{aligned}
			\tau(-\partial_{1}\Big(k_{2}^{-1}k_{1}^{-1}\partial_{1}(k_{1}k_{2})\Big))&=\tau(-\partial_{1}\Big(k_{2}^{-1}k_{1}^{-1}\partial_{1}(k_{1})k_{2}+k_{2}^{-1}\partial_{1}(k_{2})\Big))\\
			&=\tau(-\partial_{1}(k_{2}^{-1})k_{1}^{-1}\partial_{1}(k_{1})k_{2}-k_{2}^{-1}\partial_{1}\Big(k_{1}^{-1}\partial_{1}(k_{1})\Big) k_{2}\\
			&-k_{2}^{-1}k_{1}^{-1}\partial_{1}(k_{1})\partial_{1}(k_{2}))-\tau(\partial_{1}\Big( k^{-1}_{2}\partial_{1}(k_{2})\Big))\\
			&=-\tau\Big((k_{2}\partial_{1}(k_{2}^{-1})+\partial_{1}(k_{2})k_{2}^{-1})k_{1}^{-1}\partial_{1}(k_{1})\Big)\\
			&-\tau(\partial_{1}\Big(k_{1}^{-1}\partial_{1}(k_{1})\Big))-\tau(\partial_{1}\Big(k_{2}^{-1}\partial_{1}(k_{2})\Big))
		\end{aligned}	
	\end{equation*}
	But $(k_{2}\partial_{1}(k_{2}^{-1})+\partial_{1}(k_{2})k_{2}^{-1})=\partial_{1}(k_{2}k_{2}^{-1})=0$ and hence
	\begin{equation}
	\label{tau1}
	\tau(-\partial_{1}\Big(k_{2}^{-1}k_{1}^{-1}\partial_{1}(k_{1}k_{2})\Big))=-\tau(\partial_{1}\Big(k_{1}^{-1}\partial_{1}(k_{1})\Big))-\tau(\partial_{1}\Big(k_{2}^{-1}\partial_{1}(k_{2})\Big))
	\end{equation}
	With exactly similar computation, we have 
	\begin{equation}
	\label{tau2}
	\tau(-\partial_{2}\Big(k_{2}^{-1}k_{1}^{-1}\partial_{2}(k_{1}k_{2})\Big))=-\tau(\partial_{2}\Big(k_{1}^{-1}\partial_{2}(k_{1})\Big))-\tau(\partial_{2}\Big(k_{2}^{-1}\partial_{2}(k_{2})\Big))
	\end{equation}
	Combining \eqref{tau1} and \eqref{tau2}, we get that the expression \[\tau(-\partial_{1}\Big(k_{2}^{-1}k_{1}^{-1}\partial_{1}(k_{1}k_{2})\Big)-\partial_{2}\Big(k_{2}^{-1}k_{1}^{-1}\partial_{2}(k_{1}k_{2})\Big))\] equals
	\[
	\tau(-\partial_{1}\Big(k_{1}^{-1}\partial_{1}(k_{1})\Big)-\partial_{2}\Big(k_{1}^{-1}\partial_{2}(k_{1})\Big))+ \tau(-\partial_{1}\Big(k_{2}^{-1}\partial_{1}(k_{2})\Big)-\partial_{2}\Big(k_{2}^{-1}\partial_{2}(k_{2})\Big)),
	\]
	which is equal to $\int_{\tau}\Omega^{k_{1}}_{12}+\int_{\tau}\Omega^{k_{2}}_{12}$, completing the proof.  
	\end{proof}
For the proof of the next lemma see the proof of Proposition 4.1 of \cite{rosenberg}.
\blmma
Let $A$ be a $\textup{C}^*$\nobreakdash-algebra generated by two isometries with two separating derivations on $\cla$ so that it admits the canonical calculus of rank two. Moreover, suppose that $A$ has a tracial state $\tau$ which is invariant under the $\mathbb{T}^{2}$\nobreakdash-action. Then the calculus satisfies the Gauss-Bonnet theorem.
\elmma
The next corollary follows from the standard fact that the noncommutative $2$\nobreakdash-torus and the group $\textup{C}^*$\nobreakdash-algebra on the free group on two generators admit a unique $\mathbb{T}^{2}$\nobreakdash-invariant tracial state.
\bcrlre \label{13jul21sm2}
The canonical calculus of rank two on the noncommutative $2$\nobreakdash-torus and that on $\textup{C}^{\ast}_{r}(\mathbb{F}_{2})$ satisfy the Gauss-Bonnet theorem.
\ecrlre 
\brmrk
Although the Cuntz algebra with two generators does not admit a tracial state, it also satisfies the Gauss-Bonnet theorem with respect to the unique KMS state. This is because the KMS state is invariant under the toral action. The proof can be given following the same argument as in \cite[Proposition 4.1]{rosenberg}.
\ermrk 
\brmrks
\hfill
\begin{enumerate}
	\item As already mentioned in the Introduction, there is another approach to Gauss-Bonnet theorem on the noncommutative $2$\nobreakdash-torus adapted in \cite{tretkof,Sitarz,Khalkhali}. In the cited works, the geometric data is entirely encoded in the Dirac operator. More specifically, the scalar curvature is equal to $\zeta_{D}(0)$, where $\zeta_{D}(s)$ is the spectral zeta function admitting a meromorphic continuation to the whole complex plane.
	\item In \cite{Sitarz}, the Dirac operator is perturbed and shown that upto second order perturbation, the zeta function vanishes at $0$ and this statement is taken as the Gauss-Bonnet theorem.
	\item In \cite{Khalkhali}, the Dirac operator is conformally rescaled using a globally diagonalizable matrix $h$. Denoting the conformally rescaled Dirac operator by $D_{h}$, the Gauss-Bonnet theorem takes the form $\zeta_{D}(0)=\zeta_{D_h}(0)$, for all globally diagonalizable matrices $h$.
	\item In \cite{Sitarz,Khalkhali} and in our paper, the Gauss-Bonnet theorem for noncommutative 2-torus essentially takes the form ``the scalar curvature is independent of some conformal deformation parameter''. We directly deform the canonical metric in our paper, whereas in \cite{Sitarz,Khalkhali}, the metric is deformed indirectly via deforming the Dirac operator.
\end{enumerate}
\ermrks

\subsection{A class of strongly $\sigma$\nobreakdash-compatible metrics not amenable to the Gauss-Bonnet theorem} \label{13jul21sm3}


In this subsection we give an example of a class of strongly $\sigma$\nobreakdash-compatible metrics on the calculus of rank two on the noncommutative $2$\nobreakdash-torus for which the Gauss-Bonnet type theorem fails. Let us briefly mention what a Gauss-Bonnet theorem could be for a general class of strongly $\sigma$\nobreakdash- compatible metrics. Since we do not assume any positivity of ${\rm det}((g_{ij}))$, there is no surface element available in general. So instead of looking at $\int_{\tau}K dS$, one could look at $\int_{\tau}\Omega_{12}$. Hence a Gauss-Bonnet theorem would state that $\int_{\tau}\Omega_{12}$ is independent of some parameter. To that end we fix a smooth deformation parameter $k\in \textup{C}(\mathbb{T}^{2}_{\theta})$ in the sense of Remark \ref{wedge_extension}. Then the following is a strongly $\sigma$\nobreakdash-compatible metric which is a special case of Example \ref{new} for $n=2$:
	\begin{equation}\label{badmetric}
	g(\sum_{i,j=1}^{2}e_{i}\ot e_{j}a_{ij}):=ka_{11}+a_{22}.
	\end{equation}
	Plugging in $g_{11}=k,\ g_{22}=1, \ g_{12}=g_{21}=0$ in Equation \eqref{Christoffel}, we get
	 	\bppsn
	The Christoffel symbols of the Levi-Civita connection $\nabla_{g}$ are given by
	\begin{eqnarray*}
	&&\Gamma^{1}_{11}=\frac{1}{2}k^{-1}\partial_{1}(k), \Gamma^{1}_{12}=\Gamma^{1}_{21}=\frac{1}{2}k^{-1}\partial_{2}(k), \\
	&& \Gamma^{1}_{22}=0;
	\Gamma^{2}_{12}=\Gamma^{2}_{21}=\Gamma^{2}_{22}=0,  \Gamma^{2}_{11}=-\frac{1}{2}k^{-2}\partial_{2}(k).
	\end{eqnarray*}
\eppsn
Now we can deduce the expression for $\Omega_{12}^{k}$ from Equation (\ref{curvatureform}) as well as the formula for scalar curvature by plugging in the Christoffel symbols in Proposition 6.8 of \cite{conformal_bhowmick}. We state these in the form of a theorem.
\bthm
\label{express}
Let $k\in \textup{C}(\mathbb{T}^{2}_{\theta})$ be a smooth deformation parameter (in the sense of Remark \ref{wedge_extension}). Then the scalar curvature and the curvature $2$\nobreakdash-form of the Levi-Civita connection $\nabla_{g}$ ($g$ as in (\ref{badmetric})) are given by 

\[
\begin{aligned}
{\rm Scal}&=\frac{1}{8}\Big(k^{-1}\partial_{2}(k)k^{-1}\partial_{2}(k)-\partial_{2}(k)k^{-2}\partial_{2}(k)\Big)-\frac{1}{2}\Big(k\partial_{2}(k^{-2}\partial_{2}(k))+\partial_{2}(k^{-1}\partial_{2}(k))\Big)\\
&-\frac{1}{4}\partial_{1}(k^{-1}\partial_{2}(k))
\end{aligned}
\]
and
\[\Omega_{12}^{k}=e_{1}\wedge e_{2}\Big(\frac{1}{4}k^{-1}\partial_{2}(k)k^{-1}\partial_{2}(k)-\frac{1}{2}\partial_{2}\big(k^{-1}\partial_{2}(k)\big)\Big),
\] respectively.
\ethm 

The following proposition justifiably says that the Gauss-Bonnet type theorem fails for the class of the strongly $\sigma$\nobreakdash-compatible metrics considered in this subsection. In the following proposition $U_{1}, U_{2}$ are the generating unitaries for the noncommutative $2$\nobreakdash-torus. The claims of the following proposition follow from simple calculations which we leave to the reader. 
\bppsn \label{gbfails}
For the strongly $\sigma$\nobreakdash-compatible metrics $g$ with parameters $k_{1}=U_{1}$ and $k_{2}=U_{2}$, we have
\[
\int_{\tau}\Omega_{12}^{k_{1}}= 0, 
\quad \int_{\tau}\Omega_{12}^{k_{2}}=\frac{1}{4}, \text{ respectively.}
\]
\eppsn

\section{Concluding remarks} We end this article with some remarks and some future directions.

\subsection*{An example of a graph $\textup{C}^*$\nobreakdash-algebra} Let $A$ be the universal $\textup{C}^*$\nobreakdash-algebra generated by partial isometries $S_{1},S_{2}$ such that $S_{1}^{\ast}S_{1}=S_{2}S_{2}^{\ast}$ and $S_{1}S_{1}^{\ast}=S_{2}^{\ast}S_{2}$. It can be shown that $S_{1}^{\ast}S_{1}+S_{2}^{\ast}S_{2}=1$. Then $A$ admits a $\mathbb{T}^{2}$\nobreakdash-action canonically and by Lemma 3.2, there exist two separating derivations on $\cla$ such that $\partial_{i}(S_{j})=\delta_{ij}S_{i}$ and thus a Dirac triple as in the proof of Proposition 3.1.

	It can be shown following the arguments in the proof of Proposition 3.1 that $\Omega_{\cld}^{1}(\cla)\subset\cla\oplus\cla$ . For the equality, we observe that $S_{1}^{\ast}[\cld,S_{1}]=S_{1}^{\ast}S_{1}\ot\gamma_{1}$ and $S_{1}[\cld,S_{1}^{\ast}]=-S_{1}S_{1}^{\ast}$ so that $S_{1}^{\ast}[\cld,S_{1}]-S_{1}[\cld,S_{1}^{\ast}]=(S_{1}S_{1}^{\ast}+S_{1}^{\ast}S_{1})\ot\gamma_{1}=1\ot\gamma_{1}\in\Omega^{1}_{\cld}(\cla)$. Similarly, $1\ot\gamma_{2}\in\Omega^{1}_{\cld}(\cla)$. The rest of the calculations are exactly similar to case of the noncommutative torus and can be executed with minor modifications. Therefore the scalar curvature with respect to the canonical bilinear metric is again $0$.

	\blmma
	The Christoffel symbols, hence the scalar curvature, are zero for any conformally deformed metric, the smooth deformation parameter being any invertible element of the invariant subalgebra $\cla^{\mathbb{T}^2}$.
	\elmma

	\begin{proof}
	Let $k$ in $\cla^{\mathbb{T}^2}$ be an invertible element. Invariance with respect to the $\mathbb{T}^2$ implies that $\partial_{i}(k)=0$ for $i=1,2$. By the Leibniz rule. it follows that $\partial_{i}(k^{-1})=0$ too. The lemma now follows from the formulae of the Christoffel symbols and the scalar curvature for a conformally deformed metric, as in Proposition 3.1.	
	\end{proof}

     This example provides a slight enlargement of our class of examples. But for a general graph $\textup{C}^*$\nobreakdash-algebra our technique fails to produce a free bimodule of one-forms and a clear need of a computational device is indicated where the calculus is not free.
	
	\subsection*{Extension of the Gauss-Bonnet theorem} We have seen that if the canonical calculus of rank two on a $\textup{C}^*$\nobreakdash-algebra with a tracial state $\tau$ satisfies the Gauss-Bonnet theorem, the quantity $\int_{\tau}\Omega_{12}^{k}$ is equal to $0$ for all smooth deformation parameters $k$. It would be interesting to look for $\textup{C}^*$\nobreakdash-algebras necessarily with a {\bf non-tracial} state $\tau$ such that the quantity $\int_{\tau}\Omega_{12}^{k}$ is a non-zero constant for all smooth deformation parameters. Unfortunately, the Cuntz algebra with two generators cannot produce such an example. One can also change the differential calculus instead to produce such a non-zero scalar. But this again requires producing examples of tame calculus which are not free.



\end{document}